\documentclass[11pt]{article}
\usepackage{amsmath,amsfonts,amsthm,amssymb,ifthen}
\usepackage{fancybox}

\newboolean{@details}
\setboolean{@details}{true}

\def\detail#1{\ifthenelse{\boolean{@details}}{\marginpar{begin\\
details}#1\marginpar{end\\
details}}{}}

\def\todo#1{\ifthenelse{\boolean{@todoon}}{\textbf{(todo\footnote{#1})}}{}}

\newtheorem{theorem}{Theorem}[section]
\newtheorem{corollary}[theorem]{Corollary}
\newtheorem{lemma}[theorem]{Lemma}

{\begin{list}{}{\usecounter{bean}%
  \setlength{\topsep}{0em}
  \setlength{\parsep}{0em}
  \setlength{\partopsep}{0em}
  \setlength{\itemsep}{0em}
}}%
{\end{list}}

{
\begin{enumerate}}%
{\end{enumerate}}

\def\myPhiSym{\ifthenelse{\boolean{@laptop}}%
{\varphi}
{\mathbb{\Phi}}
}

\def\defeq{\stackrel{\mathrm{def}}{=}}

\def\norm#1{\left\| #1 \right\|}

\def\setof#1{\left\{#1  \right\}}

\newcommand{\ceiling}[1]{\left\lceil#1\right\rceil}

\def\trace#1{\mathrm{Tr} \left( #1 \right)}
\def\softO#1{\widetilde{\mathcal{O}} \left( #1 \right)}
\def\bigO#1{O\left(#1  \right)}

\def\setof#1{\left\{#1  \right\}}

\newdimen\pIR
\newcommand\StevesR{{\rm I\kern\pIR R}}

\def\pinv#1{{#1}^{\dagger}}

\newenvironment{fminipage}%
  {\begin{Sbox}\begin{minipage}}%
  {\end{minipage}\end{Sbox}\fbox{\TheSbox}}

\def\stretch#1#2{\textrm{st}_{#1} (#2)}

\begin{document}

\title{
A Note on Preconditioning by Low-Stretch Spanning Trees%
\thanks{
This material is based upon work supported by the National Science Foundation under Grant CCF-0634957.
Any opinions, findings, and conclusions or recommendations expressed in this material are those of the author(s) and do not necessarily reflect the views of the National Science Foundation.
}}

\author{
Daniel A. Spielman \\ 
Department of Computer Science\\
Program in Applied Mathematics\\
Yale University
\and 
Jae Oh Woo\\
Program in Applied Mathematics\\
Yale University}

\maketitle

\begin{abstract}
Boman and Hendrickson~\cite{BomanHendricksonAKPW} observed that one can solve linear systems in
  Laplacian matrices in time $\bigO{m^{3/2 + o (1)} \ln (1/\epsilon )}$ by preconditioning
  with the Laplacian of a low-stretch spanning tree.
By examining the distribution of eigenvalues of the preconditioned linear
  system, we prove that the preconditioned conjugate gradient will actually
  solve the linear system in time $\softO{m^{4/3} \ln (1/\epsilon )}$.
\end{abstract}

\section{Introduction}\label{sec:intro}

For background on the support-theory approach to solving symmetric,
diagonally dominant systems of linear equations, we refer the reader to
  one of~\cite{SupportGraph,SupportTheory,SpielmanTengLinsolve}.

Given a weighted, undirected graph $G = (V, E, w)$, we recall that 
  the Laplacian of $G$ may be defined by
\[
  L_{G} = \sum_{(u,v) \in E} w (u,v) L_{(u,v)},
\]
where $L_{(u,v)}$ is the Laplacian of the weight-1 edge from $u$ to $v$.
This is, $L_{(u,v)}$ is the matrix that is zero everywhere, except for the submatrix
  in rows and columns $\setof{u,v}$ which has form:
\[
\begin{pmatrix}
1 & -1 \\
-1 & 1
\end{pmatrix}.
\]
Note that this last matrix may be written as the outer product of the vector
  $\psi_{u} - \psi_{v}$ with itself, where we let $\psi_{u}$ denote the elementary unit
  vector with a 1 in its $u$-th component.

For a connected graph $G$,
  we recall that a spanning tree of $G$ is a connected graph $T = (V,F, w)$
  where $F$ is a subset of $E$ having exactly $n-1$ edges.
As we intend for the edges that appear in  $T$ to have the same weight
  as they do in $G$, we use the same weight function $w$.
As $T$ is a tree, every pair of vertices of $V$ is connected by a
  unique path in $T$.

For any edge $e \in E$, we now define the \emph{stretch} of $e$ 
  with respect to $T$.
Let $e_{1}, \dotsc , e_{k} \in F$ be the edges
  on the unique path in $T$ connecting the endpoints of $e$.
The \emph{stretch} of $e$ with respect to $T$ is given by
\[
\stretch{T}{e}
= w (e) \left( {\sum_{i=1}^{k} 1/w (e_{i})} \right).
\]
The stretch of the graph $G$ with respect to $T$
  was defined by Alon, Karp, Peleg, and West~\cite{AKPW} to be
\[
\stretch{T}{G}
\defeq 
\sum_{e \in E} \stretch{T}{e}.
\]
A \emph{low-stretch spanning tree} of $G$ is a graph for which
  the above quantity is reasonably small.
The best known bound on attainable stretch was obtained by
  Abraham, Bartal and Neiman~\cite{AbrahamBartalNeiman}, who present an algorithm that,
  on input a graph with $n$ vertices and $m$ edges, runs in time
  $\softO{m}$ and produces a spanning tree $T$ of stretch
  $O (m \log n \log \log n (\log \log \log n)^{3})$.

The advantage of using a spanning tree as a preconditioner is that (after a permutation) 
  one can compute an LU-factorization of the Laplacian of a tree in
  time $O(n)$, and that one can use this LU-factorization to solve
  linear systems in the Laplacian of the tree in linear time as well.
  
\section{Preconditioning}
We prove the following three results.
\begin{theorem}\label{thm:trace}
Let $G = (V,E,w)$ be a connected graph and let $T = (V,F,w)$ be a spanning
  tree of $G$.
Let $L_{G}$ and $L_{T}$ be the Laplacian matrices of $G$ and $T$, respectively.
Then,
\[
\trace{L_{G} \pinv{L_{T}}} = \stretch{T}{G},
\]
where $\pinv{L_{T}}$ denotes the pseudo-inverse of $L_{T}$.
\end{theorem}

As $T$ is a subgraph of $G$, all the nonzero eigenvalues of
 $\trace{L_{G} \pinv{L_{T}}}$ are at least $1$.
The analysis of Boman and Hendrickson~\cite{BomanHendricksonAKPW} followed
  from the fact that  the largest eigenvalue of $L_{G} \pinv{L_{T}}$
  is at most $\stretch{T}{G}$.
We use the bound on the trace to show that not too many of these
  eigenvalues are large.

\begin{corollary}\label{cor:count}
For every $t > 0$, the number of eigenvalues of
  $L_{G} \pinv{L_{T}}$ greater than $t$ is at most $\stretch{T}{G}/t$.
\end{corollary}

\begin{theorem}\label{thm:time}
If one uses the preconditioned conjugate gradient (PCG) to solve a linear equation
  in $L_{G}$ while using $L_{T}$ as a preconditioner, it will find a solution
  of accuracy $\epsilon$ in at most $\bigO{\stretch{T}{G}^{1/3} \ln (1/\epsilon )}$ iterations.
\end{theorem}

As the dominant cost of each iteration of PCG is
  the time required to multiply a vector by $L_{G}$, which is $O(m)$,
  and the time required to solve a system of equations in $L_{T}$, which is
  $O(n)$, the low-stretch spanning trees of Abraham, Bartal and Neiman enable
  PCG to run in time
\[
  \bigO{m^{4/3} (\log n)^{1/3} (\log \log n)^{2/3} (\log 1/\epsilon )}.
\]

The following lemma is the key to the proof of Theorem~\ref{thm:trace}.

\begin{lemma}\label{lem:reff}
Let $T = (V, F, w)$ be a tree, let $u, v \in V$, and
  let $x = \psi_{u} - \psi_{v}$.
Then,
\[
  x^{T} \pinv{L_{T}} x = 
  {\sum_{i=1}^{k} 1/w (e_{i})},
\]
where $e_{1}, ..., e_{k}$ are the edges on the unique simple path in $T$
  from $u$ to $v$.
\end{lemma}
\begin{proof}
The quantity $x^{T} \pinv{L_{T}} x$ is known to equal the effective resistance
  in the electrical network corresponding to $T$ in which the resistance
  of every edge is the reciprocal of its weight (see, for example, \cite{SpielmanSrivastava}).
As only edges on the path from $u$ to $v$ can contribute to the effective resistance
  in $T$ from $u$ to $v$, the effective resistance is the same as the effective resistance
  of the path in $T$ from $u$ to $v$.
As the effective resistance of resistors in serial is just the sum of their resistances,
  the lemma follows.
\end{proof}

\begin{proof}[Proof of Theorem~\ref{thm:trace}]
We compute
\begin{align*}
\trace{L_{G} \pinv{L_{T}}}
& = 
\sum_{(u,v) \in E} w(u,v) \trace{L_{(u,v)} \pinv{L_{T}} }
\\
& = 
\sum_{(u,v) \in E} w(u,v) \trace{(\psi_{u} - \psi_{v})(\psi_{u} - \psi_{v})^{T} \pinv{L_{T}}} 
\\
& = 
\sum_{(u,v) \in E} w(u,v) \trace{(\psi_{u} - \psi_{v})^{T} \pinv{L_{T}} (\psi_{u} - \psi_{v})} 
\\
& = 
\sum_{(u,v) \in E} w(u,v)   {\sum_{i=1}^{k} 1/w (e_{i})}
\\
\intertext{(where $e_{1}, \dotsc , e_{k}$ are the edges on the simple
path in $T$  from $u$ to $v$)}
& = 
\sum_{(u,v) \in E} \stretch{T}{u,v}
\\
& = 
\stretch{T}{G}.
\end{align*}
\end{proof}

\begin{proof}[Proof of Corollary~\ref{cor:count}]
As both $L_{G}$ and $L_{T}$ are positive semi-definite,
  all the eigenvalues of $L_{G} \pinv{L_{T}}$ are real and non-negative.
The corollary follows immediately.
\end{proof}

To show that the PCG will quickly solve
  linear systems in $L_{G}$ with $L_{T}$ as a preconditioner, we use the analysis
  of Axelsson and Lindskog~\cite[(2.4)]{axelssonLindskog}, which we summarize as
  Theorem~\ref{thm:al}.

\begin{theorem}\label{thm:al}
Let $A$ and $C$ be positive semi-definite matrices with the same nullspace
  such that all but $q$ of the eigenvalues of $A \pinv{C}$ lie in the interval
  $[l,u]$, and the remaining $q$ are larger than $u$.
If $b$ is in the span of $A$ and 
  one uses the Preconditioned Conjugate Gradient with $C$ as a preconditioner
  to solve the linear system $A x = b$, then after
\[
k = q + 
\ceiling{\frac{\ln (2 / \epsilon)}{2} \sqrt{\frac{u}{l}}} 
\]
iterations, the algorithm will produce a solution $x$ satisfying
\[
  \norm{x - \pinv{A} b}_{A} \leq \epsilon \norm{\pinv{A} b}_{A}.
\]
\end{theorem}
We recall that $\norm{x}_{A} \defeq \sqrt{x^{T} A x}$.
While Axelsson and Lindskog do not explicitly deal with the case in which
  $A$ and $C$ are positive-semidefinite with the same nullspace,
  the extension of their analysis to this case is immediate if one applies
  the pseudo-inverse of $C$ whenever they refer to the inverse.

\begin{proof}[Proof of Theorem~\ref{thm:time}]
As $G$ and $T$ are connected, both $L_{G}$ and $L_{T}$ have the same nullspace:
  the span of the all-1s vector.

Set
  $u = \left(\stretch{T}{G} \right)^{2/3}$ and $l = 1$.
Corollary~\ref{cor:count} tells us that $L_{G} \pinv{L_{T}}$
  has at most 
 $q = \left(\stretch{T}{G} \right)^{1/3}$ eigenvalues
  greater than $u$.
The theorem now follows from Theorem~\ref{thm:al}.
\end{proof}

\bibliographystyle{alpha}
\bibliography{precon}

\end{document}